\documentclass[reqno]{amsart}
\usepackage[T2A]{fontenc}
\usepackage[utf8]{inputenc}
\usepackage[russian,english]{babel}
\usepackage{amsthm, amsmath, amssymb, yfonts}
\usepackage{epsfig}
\usepackage{amscd}
\usepackage{comment}
\usepackage{mathrsfs}
\usepackage{graphicx}
\usepackage[hyper]{amsbib}
\usepackage{xy}
\xyoption{all}

\usepackage{ifpdf}
\ifpdf
	\makeatletter
		\g@addto@macro\Gin@extensions{,.eps}
	\makeatother
\fi

\theoremstyle{theorem}
\newtheorem{theo}{Theorem}[section]
\newtheorem{lemma}[theo]{Lemma}
\newtheorem{prop}[theo]{Proposition}

\newtheorem{cor}[theo]{Corollary}
\theoremstyle{definition}
\newtheorem{defi}[theo]{Definition}
\newtheorem{example}[theo]{Example}

\theoremstyle{remark}
\newtheorem{remark}[theo]{Remark}

\newcommand{\bbZ}{{\mathbb Z}}
\newcommand{\bbP}{{\mathbb P}}
\newcommand{\GL}{\operatorname{GL}}

\newcommand{\Cone}{\operatorname{Cone}}

\newcommand{\End}{\operatorname{End}}
\newcommand{\Hom}{\operatorname{Hom}}
\newcommand{\Kom}{\operatorname{Kom}}
\newcommand{\Repr}{\operatorname{Rep}}

\newcommand{\Ext}{\operatorname{Ext}}
\newcommand{\Tor}{\operatorname{Tor}}
\newcommand{\Sym}{\operatorname{Sym}}
\newcommand{\coh}{\operatorname{coh}}
\newcommand{\grmod}{\operatorname{grmod}}

\newcommand{\extenddiag}[2]{{\mathbf{e}(#1,#2)}}
\DeclareMathAlphabet{\mathpzc}{OT1}{pzc}{m}{it}

\CompileMatrices

\makeatletter
\def\@settitle{\begin{center}%
    \baselineskip14\p@\relax
    \bfseries
    \@title
  \end{center}%
}
\makeatother

\title{Syzygy algebras for the Segre embeddings}
\author{Igor V. Netay}

\address{HSE, Mathematical Department, Moscow, Russia}
\address{Independent University of Moscow, Moscow, Russia}
\email{i.v.netay@gmail.com}


\begin{document}
 
\begin{abstract}
	We describe the syzygy spaces for the Segre embedding~$\bbP(U)\times\bbP(V)\subset\bbP(U\otimes V)$ in terms of representations of $\GL(U)\times \GL(V)$ and 
	construct the minimal resolutions of the sheaves~$\mathscr{O}_{\bbP(U)\times\bbP(V)}(a,b)$ in~$D(\bbP(U\otimes V))$ for~$a\geqslant-\dim(U)$ and~$b\geqslant-\dim(V)$. 
	Also we prove some property of multiplication on syzygy spaces of the Segre embedding.
\end{abstract}

\maketitle

\thanks{The work was partially supported by AG Laboratory HSE, RF government grant, ag. 11.G34.31.0023, the project RFBR 10-01-00836, government grant, ag. MK-3312.2012.1, government grant, ag. MK-6612.2012.1, the project RFBR 12-01-31012, the Simons Foundation.}

\section{Introduction}

Let~$U$ and~$V$ be vector spaces over a field~$\Bbbk$,~$\dim U=m$,~$\dim V=n$. 
The field~$\Bbbk$ is supposed to have zero characteristic and to be algebraically closed.
Consider the tensor product~$U\otimes V$.
Let~$M\subset U\otimes V$ be the subset of decomposable tensors.
Note that $M\subset U\otimes V$ is a cone.
It is well known that~$\bbP(M)=\bbP(U)\times\bbP(V)\subset\bbP(U\otimes V)$.
This embedding is called the {\it Segre embedding}.
One can identify~$U\otimes V$ with the space of~$m\times n$-matrices.
Then~$M$ gets identified with the set of matrices of rank~$\leqslant 1$.
Its ideal~$\mathbb{I}(M)$ is generated by all~$2\times 2$-minors as polynomials in matrix elements.
There are relations between these generators.
Between relations there are also some relations~etc. 
If we take minimal sets of generators and relations, then they will generate so called syzygy spaces. 
Let us give a formal definition in a more general situation.

For any projective variety~$X\subset\bbP(W)$ consider the projective coordinate algebra~$A=S/\mathbb{I}(X)$ as a graded $S$-module, 
where $S=\Bbbk[W^*]$ is the algebra of polynomials on~$W$ and~$\mathbb{I}(X)$ is the homogeneous ideal of~$X$. 
There exists a free resolution
\begin{equation}
	\label{minimal resolution for A}
	\xymatrix{
		\ldots \ar[r]^{d} & 
		F_2 \ar[r]^{d} & 
		F_1 \ar[r]^{d} & 
		F_0 \ar[r] & 
		A \ar[r] & 0,
	}
\end{equation}
that is an exact sequence of free graded $S$-modules.
Let us choose a minimal set of free generators in~$F_p$ and span by them a graded vector space.
Denote by~$R_{p,q}$ its~$q$-th graded component.
Denote by~$(p)$ the shift of grading by~$p$,~i.\,e.~the adding of~$p$ to the degree of any element of a module.
Then
$$
	F_p=
	\bigoplus_{q\in\bbZ}R_{p,q}\otimes_\Bbbk S(q).
$$
The resolution is minimal if all homogeneous components of the differential~$d$ in the resolution~\eqref{minimal resolution for A} have positive degrees.
The spaces~$R_{p,q}$ for the minimal resolution are called {\it $p$-th syzygies of degree~$q$}. 
Tensor multiplication by the trivial $S$-module $\Bbbk$ annihilates all differentials in the minimal free resolution and we get
\begin{equation}
	\label{SyzEqTor}
	R_{p,q}=\left(\Tor_p^S(A,\Bbbk)\right)_q
\end{equation}
which shows that the syzygy spaces are independent on the choice of the resolution.
The space~$\left(\Tor_p^S(A,\Bbbk)\right)_q$ is the~$q$-th graded component of the graded vector space~$\Tor_p^S(A,\Bbbk)$.

In general, it is very difficult to find syzygy spaces. They are not completely described even for projective curves.

There is a natural structure of algebra on the direct sum of syzygy spaces of any projective variety. 
In paper~\cite{GorKhorRud} the syzygy algebras of the Grassmannians~$\mathrm{Gr}(2,n)$ are found.
In paper~\cite{OtPa} some properties of the Veronese embeddings are examined.
In paper~\cite{Snowden} it is shown that syzygies of the Segre embedding are generated by finite set of <<families of equations>> (i.\,e.~relations such that all syzygies can be derived by substitutions of variables) and that these families do not depend on the number of projective spaces.

In this paper we consider the Segre embedding of the product of two projective spaces.
The representation of the group~$G=\GL(U)\times\GL(V)$ on syzygy spaces of this embedding are described in Theorem~\ref{main_add_intr}. 
In remark~\ref{multiplication} we prove some property of multiplication in this algebra.

We have the action of the group~$G$ on the space~$U\otimes V$.
This action preserves the tensor rank.
Hence, the group~$G$ preserves~$M$ and its ideal~$\mathbb{I}(M)$.
From this one obtains the action of~$G$ on the minimal resolution and syzygy spaces.
The main goal of this paper is to describe the syzygy spaces of the Segre embedding and syzygy spaces of sheaves~$\mathscr{O}_{\bbP(U)\times\bbP(V)}(a,b)$ on~$\bbP(U\otimes V)$ for~$a\geqslant-\dim U$~and~$b\geqslant-\dim V$ (see~Theorem~\ref{main_geom} and definition~\ref{min sheaf res def}) as $G$-modules.

Irreducible representations of~$\GL(n)$ correspond to Young diagrams. 
Thus, irreducible representations of~$\GL(m)\times\GL(n)$ correspond to pairs of Young diagrams~$(\lambda,\mu)$. 
\begin{defi}
%	\label{definition of l and wt}
	\label{wt and l definition}
	For a Young diagram~$\lambda$ let us denote by~$l(\lambda)$ and~$\mathrm{wt}(\lambda)$ the length of the diagonal~(i.\,e. the intersection of~$\lambda$ with the set of boxes~$\{(k,k)\}$ for all~$k\in\bbZ$) and weight~(i.\,e. number of boxes in~$\lambda$).
	Denote by~$\extenddiag{\lambda}{k}$ the diagram obtained from~$\lambda$ by adding a box to the end of each of first~$k$ columns of~$\lambda$. 
	By~$\lambda'$ we denote the transposed diagram~$\lambda$.
	By~$V_\lambda$ we denote the irreducible representation of~$\GL(V)$ with the highest weight~$\lambda$.
\end{defi}

If a Young diagram~$\lambda$ consists of one row with~$k$ boxes, then we denote the representation~$V_\lambda$ by~$\Sym^kV$.

\begin{theo}
	\label{main_add_intr}
	Let~$R_{p,q}$ be the syzygy spaces of the Segre embedding~$\bbP(U)\times\bbP(V)\subset\bbP(U\otimes V)$. Then there is an isomorphism of representations of $G$: 
	\[
		R_{p,q} \cong \bigoplus_{\mathrm{wt}(\lambda)=p,\atop l(\lambda)=q-p}\left(U_{\extenddiag{\lambda}{q-p}}^*\otimes V_{\extenddiag{\lambda'}{q-p}}^*\right).
	\]
\end{theo}

This paper is organized as follows.
In section~$2$ we follow~\cite{GorKhorRud} and~\cite{Green}, and describe Koszul complex such that its cohomology groups equal syzygy spaces, and
prove some facts on expression of syzygy spaces as cohomologies of vector bundles. 
In section~$3$ we work out necessary combinatorics. 
In section~$4$ we use this combinatorics to calculate syzygies of the Segre embedding and to calculate minimal resolutions of sheaves. 
The appendices are devoted to examples.

The author is grateful to his scientific adviser A.\,L.\,Gorodentsev, to S.\,O.\,Gorchinsky, to A.\,G.\,Kuznetsov, to \`E.\,B.\,Vinberg, and to V.\,M.\,Buchstaber for useful discussions.

\section{Preliminaries}
\subsection{Koszul complex}

Let~$X\subset\bbP(W)$ be an projective variety~$X$ in a projective space.
Let~$A_q=\mathrm{H}^0(X,\mathscr{O}(q))$ be the~$q$-th graded component of the homogeneous coordinate algebra of~$X$.
Let~$\imath\colon\Lambda^pW^*\to\Lambda^{p-1}W^*\otimes\nolinebreak W^*$ be dual to the exterior product map $\Lambda^{p-1}W\otimes W\to\Lambda^pW$.
Using the multiplication~$\pi_q\colon A_1\otimes\nolinebreak A_q\to\nolinebreak A_{q+1}$, 
we define a map $d_{p,q}$ as the composition:
\[
	\xymatrix{
		\Lambda^{p-1}W^*\otimes W^*\otimes A_q \ar[r] &
		\Lambda^{p-1}W^*\otimes A_1\otimes A_q \ar[d]^-{\mathrm{Id}\otimes \pi_q} \\
		\Lambda^pW^*\otimes A_q \ar[u]^-{\imath\otimes\mathrm{Id}}\ar[r]_-{d_{p,q}} &
		\Lambda^{p-1}W^*\otimes A_{q+1},
	}
\]
where the top arrow is induced by the natural restriction map
\[
	W^* = \mathrm{H}^0(\bbP(W),\mathscr{O}(1)) \to
	\mathrm{H}^0(X,\mathscr{O}(1))=A_1.
\]

\begin{lemma}
	\label{lemma on Koszul complex of arbitrary variety}
	The chain of maps
	\begin{equation}
		\label{Koszul complex in general case}
		\xymatrix{
			\ldots \ar[r] &
			\Lambda^{p+1}W^*\otimes A_{q-1} \ar[r]^-{d_{p+1,q-1}} &
			\Lambda^pW^*\otimes A_{q} \ar[r]^-{{d_{p,q}}} &
			\Lambda^{p-1}W^*\otimes A_{q+1} \ar[r] &
			\ldots
		}
	\end{equation}
	is a complex. Moreover, its cohomologies are the syzygy spaces
	\[
		R_{p,p+q}=\frac{\ker(d_{p,q})}{\mathrm{im}(d_{p+1,q-1})}.
	\]
\end{lemma}

\begin{proof}
Consider the Koszul complex~$\Lambda^\bullet W^*\otimes\Sym^\bullet W^*=\Lambda^\bullet W^*\otimes S$ (see~\cite{GelMan89},~Ch.\,1). 
It is quasi-isomorphic to~$\Bbbk$, the trivial~$S$-module. 
Tensoring it over~$S$ by~$A$, we obtain a quasi-isomorphism~$\Lambda^\bullet W^*\otimes A\cong\Bbbk\mathop{\otimes}\limits_S^L A$. Consequently, the chain
\[
	\ldots\to\Lambda^{p+1}W^*\otimes A\to
	\Lambda^pW^*\otimes A\to
	\Lambda^{p-1}W^*\otimes A\to\ldots
\]
is a complex and its cohomology spaces are graded vector spaces~$\Tor_p^S(\Bbbk,A)$.

Since the differential~$d$ of the complex~\eqref{Koszul complex in general case} is homogeneous and has degree~$0$, we can decompose the Koszul complex into a sum of the subcomplexes:
  \[
	\xymatrix{
		\ldots \ar[r] &
		\Lambda^{p+1}W^*\otimes A_{q-1} \ar[r]^-{d_{p+1,q-1}} &
		\Lambda^pW^*\otimes A_{q} \ar[r]^-{{d_{p,q}}} &
		\Lambda^{p-1}W^*\otimes A_{q+1} \ar[r] &
		\ldots
	}
  \]
  
Finally, we get
$
	R_{p,p+q}=(\Tor_p^S(A,\Bbbk))_{p+q}=\frac{\ker(d_{p,q})}{\mathrm{im}(d_{p+1,q-1})}.
$
\end{proof}

\subsection{Projective coordinate algebras}

Let~$G$ be a reductive algebraic group.
Let~$W=V_{\lambda}$ be the irreducible representation of~$G$ with the highest weight~$\lambda$.
Let~$X$ be the $G$-orbit of the point~$w\in\bbP(W)$ corresponding to the highest weight vector in the representation~$W$.
In~\cite{LanTow} it is proved that such a variety~$X$ is an intersection of quadrics in~$\bbP(W)$.
  
\begin{remark}
	Recall that $X=G\cdot w\cong G/P$ is a projective variety, where~$P$ is a parabolic subgroup.
	Let~$I$ be the set of simple roots orthogonal to~$\lambda$.
	Each set~$I$ of simple roots corresponds to a parabolic subgroup~$P_I\subset G$ containing a fixed Borel subroup~$B$ of~$G$.
	Recall that homogeneous line bundles on~$G/P_I$ correspond to weights of~$G$ orthogonal to all roots in~$I$.
	In particular,~$\mathscr{O}_X(1)=\mathscr{O}_{\bbP(W)}(1)|_X$ corresponds to the weight~$\lambda$.
\end{remark}

The following Proposition with more detailed proof and other results on highest weight orbits and more general quasi-homogeneous spaces can be found in~\cite{VinPop72}.

\begin{prop}
	In the above notations, the projective coordinate algebra of $X$ is $$A_X=\bigoplus_{n\geqslant 0}V_{n\lambda}^*.$$
\end{prop}

\begin{proof}
By the Borel--Bott--Weyl Theorem (see~\cite{FH}) we get
$$
	A_X=\bigoplus_{n\geqslant0}\Gamma(X,\mathscr{O}_X(n))=
	\bigoplus_{n\geqslant0}\Gamma(G/P,\mathscr{L}_{n\lambda})=
	\bigoplus_{n\geqslant0}\Gamma(G/B,\mathscr{L}_{n\lambda})=
	\bigoplus_{n\geqslant0}V_{n\lambda}^*.
$$
\end{proof}

\begin{cor}
	\label{last_complex}
	In the above notations, the complex
	\[
		\ldots\to\Lambda^{p+1}W^*\otimes_\Bbbk V_{(q-1)\lambda}^*\to\Lambda^pW^*\otimes_\Bbbk V_{q\lambda}^*\to\Lambda^{p-1}W^*\otimes_\Bbbk V_{(q+1)\lambda}^*\to\ldots
	\] 
	of representations of~$G$ computes the syzygy spaces of $X=G\cdot w\subset \bbP(W)$.
\end{cor}

Denote by $\Sigma_\lambda$ the Schur functor~$\mathrm{Vect}\to\mathrm{Vect}$. 
This is a polynomial functor on the category of vector spaces, such that for $G=\GL(V)$ the representation $\Sigma_\lambda V$ is the unique representation of~$G$ with the highest weight $\lambda$. 
The symbol~$\Sigma_\lambda V'$ means the same as~$V_\lambda$ if~$V'=V$ is the tautological representation of~$\GL(V)$.
The construction of the Schur functor can be found in~\cite{Fulton}.

\begin{cor}
	Let~$W_i=\Sigma_{\lambda_i}V_i$ and~$W_1\otimes\ldots\otimes W_m$ be an irreducible $\GL(V_1)\times\ldots\times\GL(V_m)$-module and 
	$X\subset\bbP(W)$ be an orbit of highest weight vector $\lambda_1\oplus\ldots\oplus\lambda_m$ in~$W=W_1\otimes\ldots\otimes W_m$. 
	Then the projective coordinate algebra of $X$ equals 
	\[
		A_X=\bigoplus_{n\geqslant 0}\Sigma_{n\lambda_1}V_1^*\otimes\ldots\otimes\Sigma_{n\lambda_m}V_m^*
	\] 
	as a $\GL(V_1)\times\ldots\times\GL(V_m)$-module. In particular the Koszul complex~$\Lambda^\bullet(W_1\otimes\ldots\otimes W_n)\otimes A_X$
	computes the syzygy spaces of~$X$.
\end{cor}

\subsection{Lie algebra cohomology}
Let $A$ be a graded algebra. 
It is called {\it linearly generated} if~$A_0=\Bbbk$ and the natural map~$\mathrm{T}^\bullet(A_1)\to A$ is surjective. 
Linearly generated algebra~$A$ is called {\it quadratic} if the kernel~$J_A$ of the map~$\mathrm{T}^\bullet(A_1)\to A$ is generated as a two-sided 
ideal in~$\mathrm{T}^\bullet(A_1)$ by its subspace~$I_A=J_A\cap({A_1}\otimes{A_1})\subset A_1\otimes A_1$. 
Denote by $(A_1,I_A)$ the algebra $\mathrm{T}^\bullet(A_1)/(I_A)$. 
The {\it quadratic dual algebra}~$A^!$ is defined by the pair $(A_1^*,I_A^\bot)$.

A quadratic algebra is called a {\it Koszul algebra} if $A\simeq \Ext_{A^!}^\bullet(\Bbbk,\Bbbk)$.
It is well known that projective coordinate algebras of highest weight orbits in irreducible representations of reductive groups are Koszul algebras~(see~\cite{PP}).

The algebra~$A^!$ is the universal enveloping for the graded Lie super-algebra
\[
	L=\bigoplus_{m\geqslant1}L_m=\mathscr{L}ie(A_1^*)/(I_A^\bot),
\]
where~$\mathscr{L}ie(V)$ is the free graded Lie algebra generated by the vector space~$V$.
\begin{prop}
Let the projective coordinate algebra~$A_X$ of~$X$ be a Koszul algebra. Then there is an isomorphism of algebras
$$
	R\simeq\mathrm{H}^\bullet(L_{\geqslant2},\Bbbk),\quad R_{p,q}\simeq\left(\mathrm{H}^{q-p}(L_{\geqslant2},\Bbbk)\right)_q.
$$
\end{prop}

Here~$\mathrm{H}^\bullet(L_{\geqslant2},\Bbbk)$ denotes Lie algebra cohomologies.
This proposition is proved in~\cite{GorKhorRud}. In this way the syzygies for the Pl\"ucker embeddings of Grassmannians~$\mathrm{Gr}(2,n)$ are found.

\begin{example}
	Consider~$X=\mathrm{Gr}(2,n)\subset\bbP(\Lambda^2 V)$, where~$\dim V=n$. Then from~\cite{GorKhorRud} we get~
	\[
		R_{p,q}=\bigoplus_{{\mathrm{wt}(\lambda)=p}\atop{\lambda=(i_1,\ldots,i_q|i_q+3,\ldots,i_1+3)}}V_\lambda^*.
	\]
	Here~$(a_1,\ldots,a_k|b_k,\dots,b_1)$ is the Young diagram composed of embedded hooks with pairs of length and height~$(a_1,b_1)$,~$\ldots$,~$(a_k,b_k)$
	so that its weight is~$\mathrm{wt}(\lambda)=\sum_{i=1}^n(a_i+b_i)-k$.
\end{example}

\subsection{Resolutions of sheaves}
Denote by~$\coh(X)$ the category of coherent sheaves on~$X$ and by~$\grmod(S)$ the category of finitely generated graded~$S$-modules.
Let~$X\subset\bbP^N$ be a projective variety. We define a functor~$F\colon\coh(X)\to\grmod(S)$ by the formula
\[
	F(\mathscr{F})=
	\bigoplus_{n\geqslant0}\Gamma\left(\bbP^N,\mathscr{F}(n)\right)(-n)=
	\Hom\left(
		\bigoplus_{n\geqslant0}\mathscr{O}(-n),-
	\right).
\]
Clearly, one can extend it to the right derived functor~$RF\colon D^b\coh(X)\to D^b(\grmod(S))$.
\begin{lemma}
	One has
	\[
		R^iF(\mathscr{F})=
		\bigoplus_{n\geqslant0}\mathrm{H}^i\left(\bbP^N,\mathscr{F}(n)\right)(-n)=
		\Ext^i\left(
			\bigoplus_{n\geqslant0}\mathscr{O}(-n),-
		\right).
	\]
\end{lemma}

\begin{proof}
	By definition for the $n$-th graded component of the functor~$F$ we have~$F_n(\mathscr{F})=\Hom(\mathscr{O}(-n),\mathscr{F})$.
	Therefore~$R^iF_n(\mathscr{F})=\Ext^i(\mathscr{O}(-n),\mathscr{F})$, so~$R^iF(\mathscr{F})=\bigoplus_{n\geqslant 0}\Ext^i(\mathscr{O}(-n),\mathscr{F})$.
\end{proof}

\begin{lemma}
	\label{FOmega^n(n)}
	If~$\mathrm{H}^k(X,\mathscr{F}(n))=0$ for~$n\geqslant 0$ and~$k>0$ then~$\mathrm{H}^k(X,\mathscr{F}\otimes\Omega^i(i+n))=0$ for~$n\geqslant 0$,~$k>0$ and~$i>0$.
\end{lemma}

\begin{proof}
	Consider the exact sequence
	\[
		0 \to
		\Omega^i(i) \to
		\Lambda^iW^*\otimes\mathscr{O}\to
		\Lambda^{i-1}W^*\otimes\mathscr{O}(1)\to
		\ldots\to
		\mathscr{O}(i)\to
		0.
	\]
	Tensor with~$\mathscr{F}(n)$ and look at the hypercohomology spectral sequence.
\end{proof}

\begin{prop} 
	\label{Koszul_lemma}
	Assume~$\mathscr{F}\in D^b\coh(X)$ is a sheaf on a smooth projective variety~$X\subset\bbP(W)=\bbP^N$ 
	such that for all~$n>0$ and~$k\geqslant0$ we have~$\mathrm{H}^n(X,\mathscr{F}(k))=0$. 
	Then for all~$p$ and~$q$
	\[
		\mathrm{H}^{q-p}\left(\bbP^N,\Omega_{\bbP(W^*)}^p(p)\otimes\mathscr{F}\right)
		\cong
		\left(\Tor_p^S(F(\mathscr{F})),\Bbbk\right)_q.
	\]
\end{prop}

\begin{proof}
	Let~$\Delta\colon\bbP^N\to\bbP^N\times\bbP^N$ be the diagonal embedding.
	Consider the diagram
	\[
		\xymatrix{
			\bbP^N \ar[r]^-\Delta & \bbP^N\times\bbP^N \ar[dl]_-{p_1} \ar[dr]^-{p_2} & \\
			\bbP^N && \bbP^N \\
		}
	\]
	and the standard resolution of the diagonal 
	\begin{equation} 
		\label{diag_resolution}
		0 \to
		\Omega^N(N)\boxtimes\mathscr{O}(-N) \to
		\ldots \to
		\Omega(1)\boxtimes\mathscr{O}(-1) \to
		\mathscr{O}\to 0,
		%\Delta_*\mathscr{O}_{\bbP^N}\to 0.
	\end{equation}
	where~$\mathscr{G}\boxtimes\mathscr{H}$ denotes~$p_1^*\mathscr{G}\otimes p_2^*\mathscr{H}$ for sheaves~$\mathscr{G}$ and~$\mathscr{H}$ on~$\bbP^N$.
	Let~$D_k$ be stupid truncation of~\eqref{diag_resolution}:
	\[
		D_k=\left\{
		\Omega^{N-k}(N-k)\boxtimes\mathscr{O}(k-N) \to
		\ldots \to
		\Omega(1)\boxtimes\mathscr{O}(-1) \to
		\mathscr{O}\right\}.
	\]
	Then we have a chain of morphisms of complexes
	\[
		\xymatrix{
			0=D_{N+1} \ar[r] &
			D_N \ar[r] &
			\ldots \ar[r] &
			D_1 \ar[r] &
			D_0=\Delta_*\mathscr{O}_{\bbP^N},
		}
	\]
	such that~$\Cone(D_{N-i+1}\to D_{N-i})\cong \Omega^i(i)\boxtimes\mathscr{O}(-i)[i]$. 

	Tensoring with~$p_2^*\mathscr{F}$ and applying~$R{p_1}_*$, we obtain a chain of morphisms of complexes
	\begin{multline*}
		0=R{p_1}_*(p_2^*\mathscr{F}\otimes D_{N+1})\to R{p_1}_*(p_2^*\mathscr{F}\otimes D_{N})\to\ldots\\\ldots\to R{p_1}_*(p_2^*\mathscr{F}\otimes D_1)\to R{p_1}_*(p_2^*\mathscr{F}\otimes D_0)=\mathscr{F}
	\end{multline*}
	such that
	\[
		\Cone(R{p_1}_*(p_2^*\mathscr{F}\otimes D_{N-i+1})\to R{p_1}_*(p_2^*\mathscr{F}\otimes D_{N-i}))=\mathrm{H}^\bullet(\bbP^N,\mathscr{F}\otimes\Omega^i(i))\otimes\mathscr{O}(-i)[i].
	\]
	Applying~$RF$ and taking into account that~$R^iF(\mathscr{F})=0$ for~$i>0$.~i.\,e.~$RF(\mathscr{F})=F(\mathscr{F})$ and~$RF(\mathscr{O}(-p))=S(p)$, 
	we conclude that we have a chain of morphisms 
	\begin{multline*}
		0=RFR{p_1}_*(p_2^*\mathscr{F}\otimes D_{N+1})\to RFR{p_1}_*(p_2^*\mathscr{F}\otimes D_{N})\to\ldots\\\ldots\to RFR{p_1}_*(p_2^*\mathscr{F}\otimes D_1)\to RFR{p_1}_*(p_2^*\mathscr{F}\otimes D_0)=F(\mathscr{F})
	\end{multline*}
	such that
	\[
		\Cone(RFR{p_1}_*(p_2^*\mathscr{F}\otimes D_{N-i+1})\to RFR{p_1}_*(p_2^*\mathscr{F}\otimes D_{N-i}))=\mathrm{H}^\bullet(\bbP^N,\mathscr{F}\otimes\Omega^i(i))\otimes S(i)[i].
	\]

	Tensoring with~$\Bbbk$, we obtain a chain of morphisms of complexes
	\begin{multline*}
			0=RFR{p_1}_*(p_2^*\mathscr{F}\otimes D_{N+1})\mathop{\otimes}\limits_S^L\Bbbk \to
			RFR{p_1}_*(p_2^*\mathscr{F}\otimes D_{N})\mathop{\otimes}\limits_S^L\Bbbk \to
			\ldots\\\ldots\to
			RFR{p_1}_*(p_2^*\mathscr{F}\otimes D_{1})\mathop{\otimes}\limits_S^L\Bbbk \to
			RFR{p_1}_*(p_2^*\mathscr{F}\otimes D_0)\mathop{\otimes}\limits_S^L\Bbbk=F(\mathscr{F})\mathop{\otimes}\limits_S^L\Bbbk
	\end{multline*}
	such that
	\[
		\Cone\left(RFR{p_1}_*(p_2^*\mathscr{F}\otimes D_{N-i+1})\mathop{\otimes}\limits_S^L\Bbbk\to RFR{p_1}_*(p_2^*\mathscr{F}\otimes D_{N-i})\mathop{\otimes}\limits_S^L\Bbbk\right)=\mathrm{H}^\bullet(\bbP^N,\mathscr{F}\otimes\Omega^i(i))\otimes \Bbbk(i)[i].
	\]
	Now we consider the spectral sequence of a filtered complex applied to~$RFR{p_1}_*(p_2^*\mathscr{F}\otimes D_0)\mathop{\otimes}\limits_S^L\Bbbk=F(\mathscr{F})\mathop{\otimes}\limits_S^L\Bbbk$. 
	Its first term is
	\begin{equation}
		\label{main spectral sequence}
		E_1^{r,s}=\bigoplus_{n\geqslant 0}\mathrm{H}^s(\bbP^N,\mathscr{F}\otimes\Omega^{-r}(-r+n))\otimes\Bbbk(r+n).
	\end{equation}
	By Lemma~\ref{FOmega^n(n)} we get~$E_1^{r,s}=0$ for~$s>0$.
	Hence, the spectral sequence degenerates and~$\Tor_p(F(\mathscr{F}),\Bbbk)$ has a filtration 
	with factors being~$E_1^{r,s}$ with~$r+s=p$. 
	Thus,~$\left(\Tor_p(F(\mathscr{F}),\Bbbk)\right)_q=\mathrm{H}^{q-p}(\bbP^N,\mathscr{F}\otimes\Omega^p(p))$.
\end{proof}

\begin{defi}
	\label{min sheaf res def}
	Let~$\mathscr{F}$ be a sheaf on~$\bbP^N$. There exists a spectral sequence with the first term
	\[
		E_1^{-p,q}=\mathrm{H}^q(\mathbb{P}^N,\mathscr{F}\otimes\Omega^p(p))\otimes\mathscr{O}(-p),
	\]
	converging to~$\mathscr{F}$.
	Vector spaces~$R_{p,q}(\mathscr{F})=\mathrm{H}^q(\mathbb{P}^N,\mathscr{F}\otimes\Omega^p(p))$ are called {\it syzygy spaces} of~$\mathscr{F}$.
\end{defi}

Note that in the case~$\mathrm{H}^i(X,\mathscr{F}(k))=0$ for~$i>0$ and~$k\geqslant 0$ the spectral sequence degenerates and we get a resolution for~$\mathscr{F}$.
It follows from Lemma~\ref{Koszul_lemma} that the syzygies of a projective variety~$X$ coincide with the syzygies of its structure sheaf~$\mathscr{O}_X$, if~$X$ is smooth and~$\mathrm{H}^i(X,\mathscr{O}_X(k))=0$ for~$k\geqslant 0$ and~$i>0$.

Let us illustrate Proposition~\ref{Koszul_lemma} by the following well-known example. 
This example is calculated in~\cite{GorKhorRud} by means of quadratical dual algebra and in~\cite{Ha} by Koszul complex of full intersection.

\begin{example}
	Let~$X=\bbP^1=\bbP(V)\subset\bbP(\Sym^NV)=\bbP^N$ be the Veronese embedding.
	It is easy to check that~$\Omega_{\bbP^N}(1)|_{\bbP^1}\simeq\mathscr{O}(-1)^{\oplus N}$. 
	Then 
	\[
		\Omega_{\bbP^N}^p(p)|_{\bbP^1}=
		\Lambda^p(\Omega_{\bbP^N}(1))|_{\bbP^1}=
		\Lambda^p(\Omega_{\bbP^N}(1)|_{\bbP^1})\simeq
		\Lambda^p\mathscr{O}(-1)^{\oplus N}\simeq
		\mathscr{O}(-p)^{\oplus\binom{N}{p}}.
	\] 
	From Proposition~\ref{Koszul_lemma} we get for~$p>0$
	\[
		\dim R_{p,p+1}=
		\dim\mathrm{H}^1(\bbP^1,\Omega_{\bbP^N}^p(p))=
		\binom{N}{p}\dim\mathrm{H}^1(\bbP^1,\mathscr{O}(-p))=
		(p-1)\binom{N}{p}
	\]
	and~$\dim R_{0,0}=1$. All other~$R_{p,q}$ vanish.
\end{example}

Consider the Segre embedding~$X=\bbP(U)\times\bbP(V)\subset\bbP(U\otimes V)$ and take~$\mathscr{F}=\mathscr{O}_X(a,b)$.
Then~$\mathrm{H}^i(X,\mathscr{F}(k))=0$ for all~$i>0$ and~$k\geqslant0$, if~$a\geqslant -m$ and~$b\geqslant -n$.
In this case the syzygies of~$\mathscr{O}(a,b)$ can be computed as cohomologies of the complex
\begin{equation}
	\label{K definition}
	K^\bullet = \Lambda^\bullet(U^*\otimes V^*)\otimes\bigoplus_{k\geqslant0}\Sym^{a+k}U^*\otimes\Sym^{b+k}V^*.
\end{equation}
Thus, we obtain
\[
	R_{p,q}(\mathscr{O}(a,b))=
	\left(
		\mathrm{H}^p(K^\bullet)
	\right)_{q}.
\]

\section{Notation and combinatorics}
\subsection{Combinatorial cubes}

Let~$\mathpzc{A}$ be a~$\Bbbk$-linear abelian category. 
An object~$M\in\mathpzc{A}$ is {\it simple} if~$\Hom(M,M)=\Bbbk$. 
For example, any irreducible representation of a reductive group~$G$ is a simple object in the category~$\Repr(G)$ of representations of the group~$G$. 
Results of this section will be applied below for the case~$\mathpzc{A}=\Repr(\GL(U)\times\GL(V))$.

\begin{defi}
	A complex~$K^\bullet\in\Kom(\mathpzc{A})$ is said to be a {\it combinatorial $n$-cube ($k$-th clipped combinatorial $n$-cube)} if 
	it is isomorphic to~$M\otimes\Lambda^\bullet E$ (respectively $M\otimes\Lambda^{\geqslant k} E$) for
	a simple object~$M\in\mathpzc{A}$ and a vector space~$E$,
	where the differential~$d$ on~$M\otimes\Lambda^\bullet E$ is induced by the differential~$d(\xi)=\xi\wedge\xi_0$ on~$\Lambda^\bullet E$ for some non-zero~$\xi_0\in E$.
\end{defi}

Choosing a basis in~$E$ one identifies~$K^\bullet$ with the total complex of the tensor product~$\bigotimes_{i=1}^n\{M\xrightarrow{\xi_{0,i}}M\}$.
In particular, if~$\xi_0\ne 0$ then at least one~$\xi_{0,i}$ is an isomorphism, hence,~$\mathrm{H}^\bullet(K^\bullet)=0$.

Obviously,~if~$K^\bullet$ is a combinatorial $n$-cube, then~$K^{\geqslant k}$ is~$k$-th clipped combinatorial $n$-cube.

\begin{remark}
	Actually there is a natural combinatorial interpretation such that complexes~$K^\bullet$ that are used in this paper and that are clipped combinatorial cubes get identified with complexes of the form~$M\otimes\bigoplus_\mathcal{K}\Bbbk$, where~$\mathcal{K}$ is the intersection of the standard cube in a lattice with coordinates~$(x_1,\ldots,x_n)$ with the half-space~$\sum_{i=1}^kx_k\geqslant k$.
	The degree of an element corresponding the point~$(x_1,\ldots,x_n)$ is~$\sum_{i=1}^kx_k$.
	Therefore clipping of the cube corresponds to the stupid truncation of the complex corresponding to the standard cube.
	The differential takes an element corresponding to~$(x_1,\ldots,x_k)$ to the sum of elements corresponding to~$(x_1,\ldots,x_i+1,\ldots,x_n)$ for~$i=1,\ldots,n$.
	This explains why we use the notion of combinatorial cube.
\end{remark}

\begin{lemma}
	If~$K^\bullet$ is a combinatorial~$n$-cube with~$n>0$, then~$\mathrm{H}^i(K^\bullet)=0$ for all~$i$.
	If~$n=0$, then~$\mathrm{H}^0(K^\bullet)=M$ and~$\mathrm{H}^i(K^\bullet)=0$ for~$i>0$.
	If~$K^\bullet$ is a~$k$-th clipped combinatorial $n$-cube, then~$\mathrm{H}^k(K^\bullet)=M^{\oplus\binom{n-1}{k}}$ and $\mathrm{H}^i(K^\bullet)=0$ for~$i\ne k$ and~$0<k\leqslant n$.
\end{lemma}

\begin{proof}
	Consider combinatorial~$n$-cube~$L^\bullet$ based on the same object~$M$, vector space~$E$ and a vector~$\xi_0\in E$.
	Then~$K^\bullet=L^{\geqslant k}$.
	Let us decompose~$E=E'\oplus\langle \xi_0\rangle$.
	Then~$\Lambda^kE = \Lambda^kE'\oplus\Lambda^{k-1}E'\wedge\xi_0$ for any~$k$.
	Thus, rank of~$d\colon\Lambda^kE\to\Lambda^{k+1}E$ equals~$\dim\Lambda^kE'=\binom{n-1}{k}$.
	Hence, we get the statement of the Lemma.
\end{proof}

\begin{remark}
	\label{remark about equal cubes}
	If~$K$ and~$K'$ are $k$-th clipped combinatorial $n$-cubes based on the same simple object~$M$, then~$K\cong K'$.
	Thus, if~$K$ and~$K'$ are~$k$-th and~$k'$-th clipped combinatorial~$n$-cubes based on~$M$ and~$k'<k$, then~$K \simeq (K')^{\geqslant k}$.
	Any morphism~$K\to K'$ is an isomorphism onto~$(K')^{\geqslant k}$ if~$k'\leqslant k$.
\end{remark}

\begin{lemma} \label{clipped_cube_lemma}
	Let~$(K^\bullet,d)$ be a complex over an abelian category~$\mathpzc{A}$,~$E=\langle\xi_1\ldots,\xi_n\rangle$ be a vector space. 
	Suppose~$K^{\bullet}= M\otimes\Lambda^{\geqslant k}E$ as graded objects, where~$M$ is a simple object of~$\mathpzc{A}$.
	As an operator on~$M\otimes\Lambda^{\geqslant k}E$, the differential~$d$ is defined by matrix elements~$d_I^J\in\Bbbk$ such that
	\begin{equation} \label{clipped cube basis}
		d\left(m\otimes\bigwedge_{i\in I}\xi_i\right)=
		\sum_{|J|=|I|+1}\left(d_I^J\cdot m\otimes\bigwedge_{j\in J}\xi_j\right).
	\end{equation}
	Suppose~$d_I^J\ne0$ iff~$I\subset J$ and~$|I|+1=|J|$. Then~$K^\bullet$ is a $k$-th clipped combinatorial $n$-cube.
\end{lemma}

\begin{proof}
	We carry out the proof by induction on~$n$. 
	The case~$n=1$ is trivial. 
	Suppose $n>1$, then~$E=E'\oplus \Bbbk$, where~$E'=\langle \xi_1, \ldots, \xi_{n-1}\rangle$ and~$\Bbbk=\langle \xi_n \rangle$.
	The complex~$K^\bullet = K_1^\bullet\oplus K_2^\bullet$, where~$K_1^\bullet = M\otimes \Lambda^{\geqslant k} E'$ and~$K_2^\bullet=M\otimes\Lambda^{\geqslant k-1}E' \wedge \xi_n$.
	By the induction hypothesis~$K_1^\bullet\simeq M\otimes\Lambda^{\geqslant k}E'$ with a differential~$d_1=\xi\wedge-$ for~$\xi\in E'$ 
	and~$K_2^\bullet\simeq M\otimes\xi_n\wedge\Lambda^{\geqslant k-1}E'$ with a differential~$d_2=\xi'\wedge-$ for~$\xi'\in E'$.
	By Remark~\ref{remark about equal cubes} we can assume and~$\xi=\xi'$.
	Denote~$d'=\xi\wedge-$ on~$K^\bullet$.

	Then the differential~$d$ on~$K^\bullet$ is given by the formula~$d(-)=d'(-)+\varphi(-)$
	for some homogeneous homomorphism~$\varphi \colon M\otimes\Lambda^{\geqslant k} E'\to M\otimes\xi_n\wedge\Lambda^{\geqslant k-1} E'$ of degree~$1$.
	By~\eqref{clipped cube basis} we have~$\varphi(m)=cm\otimes\xi_n$ for some~$c\ne 0$.
	Therefore by Remark~\ref{remark about equal cubes} we obtain~$\varphi=c\xi_n\wedge-$.
	We see that the complex~$K^\bullet$ is isomorphic to~$M\otimes\Lambda^{\geqslant k}E$ with the differential~$d=\xi''\wedge-$, where~$\xi''=\xi+\xi_n$.
\end{proof}

\subsection{Diagrams}

Let $\bbP(U)\times\bbP(V)\subset\bbP(U\otimes V)$ be the Segre embedding,~$\dim U=m$,~$\dim V=n$,~$U\otimes V=W$.
Let us denote
$$
	S_U=\bigoplus_{i\geqslant0}\Gamma\left(\bbP(U),\mathscr{O}_{\bbP(U)}(i)\right) = \Sym^\bullet(U^*),
$$
$$
	S_V=\bigoplus_{i\geqslant0}\Gamma\left(\bbP(V),\mathscr{O}_{\bbP(V)}(i)\right) = \Sym^\bullet(V^*),
$$
$$
	A=\bigoplus_{i\geqslant0}\Gamma\left(\bbP(U)\times\bbP(V),\mathscr{O}_{\bbP(U\otimes V)}(i)\right) = \bigoplus_{i\geqslant0}\Sym^i(U^*)\otimes\Sym^i(V^*).
$$
We have the actions of~$\GL(U)=G_U$ on the space~$S_U$, of~$\GL(V)=G_V$ on the space~$S_V$, of~$G_U\times G_V=G$ on the space~$A$.

Let us fix some notations.
We will draw Young diagrams down and to the right from~$(0,0)$.
Recall that the length of the diagonal~$l(\lambda)$ and the weight~$\mathrm{wt}(\lambda)$ for a diagram~$\lambda$ are defined in~\ref{wt and l definition}.
By~$\lambda'$ we denote the transposed diagram~$\lambda$.

\begin{defi}
	\label{e definition}
	Let $\lambda$ be a Young diagram.
	Denote by $\extenddiag{\lambda}{n}$ the diagram obtained from~$\lambda$ by adding a box to the end of each of the first~$n$ columns.
\end{defi}
An example is given on the picture below.

\begin{center}
	\includegraphics{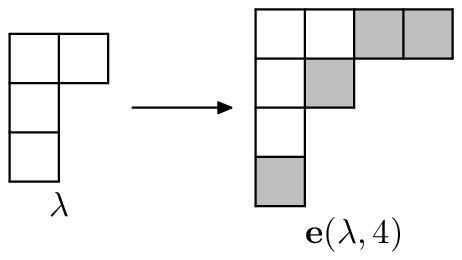}
\end{center}

Note that if the first column of~$\lambda$ has at least~$\dim V$ boxes, then the corresponding representation~$V_{\extenddiag{\lambda}{n}}=0$. 
Otherwise~$\extenddiag{\lambda}{n}$ corresponds to the representation of the minimal weight in the decomposition of~$V_\lambda\otimes\nolinebreak \Sym^n(V)$ by the Pieri formula~(see~\cite{Fulton}). 

%\begin{defi}
%	\label{wt and l definition}
%	Denote by~$\mathrm{wt}(\lambda)$ the weight of~$\lambda$~(i.e. number of boxes), by~$l(\lambda)$ the length of the main diagonal of~$\lambda$ and by~$\lambda'$ the transposed diagram~$\lambda$.
%\end{defi}

Irreducible representations of~$G=\GL(U)\times\GL(V)$ correspond to pairs of Young diagrams~$(\lambda,\mu)$.
We will intersect and subtract diagrams as sets of boxes.
Let~$(\lambda,\mu)$ be a pair of Young diagrams. 

\begin{defi}
	\label{B definition}
	Denote by~$B(\lambda,\mu)$ the set of boxes~$c$ of the diagram~$\lambda\cap\mu'$ such that
	\begin{itemize}
		\item there are no boxes of~$\lambda$ below~$c$ and
		\item there are no boxes of~$\mu'$ to the right from~$c$.
	\end{itemize}
	For a pair of diagrams~$\omega=(\lambda,\mu)$ let us write~$B(\omega)=B(\lambda,\mu)$ for brevity.
\end{defi}

On the picture below the set~$B(\lambda,\mu)$ is denoted by dots.

\begin{center}
	\includegraphics{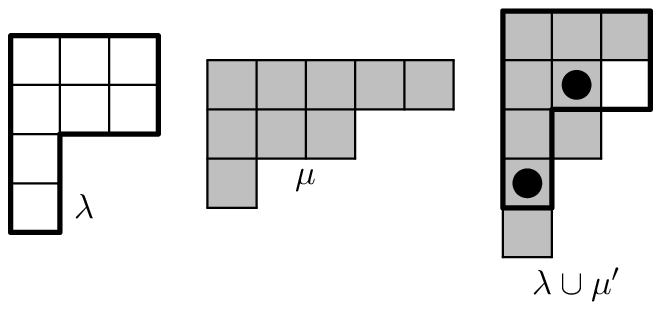}
\end{center}

\begin{defi}
	\label{Y definotion}
	Denote by~$Y(w,l,r)$ the set of Young diagrams~$\theta$ with some of the boxes marked with a letter~`L' or~`R' or with both and such that
	\begin{itemize}
		\item there is at most one~`L' in a column and at most one~`R' in a row;
		\item the total number of letters~`L' is~$l$ and the total number of letters~`R' is~$r$;
		\item there are precisely~$w$ non-marked boxes;
		\item $\theta_0$, $\theta_0\cup\theta_L$ and $\theta_0\cup\theta_R$ are Young diagrams, 
		where~$\theta_0$,~$\theta_L$ and~$\theta_R$ denote the unions of non-marked boxes of~$\theta$, boxes containing~`L' and boxes containing with~`R'.
	\end{itemize}
	Denote by~$Y_{m,n}(w,l,r)$ the subset of~$Y(w,l,r)$ that consists of~$\theta$ such that the height of~$\theta_0\cup\theta_L$ does not exceed~$m$ 
	and the width of~$\theta_0\cup\theta_R$ does not exceed~$n$.
	
	We can say that a marked diagram~$\theta$ is a triple~$(\theta_0,\theta_L,\theta_R)\in Y_{m,n}(w,l,r)$ of Young diagram~$\theta_0$ and two skew Young diagrams~$\theta_L$ and~$\theta_R$.
	Let us denote
	
	$
		\hfill\hfill\lambda(\theta)=\theta_0\cup\theta_L,\hfill\mu(\theta)=(\theta_0\cup\theta_R)',\hfill\omega(\theta)=(\lambda(\theta),\mu(\theta)).\hfill\hfill
	$
\end{defi}

The picture below gives an example of an element of~$Y_{4,4}(6,2,3)$.
\begin{center}
	\includegraphics[scale=1]{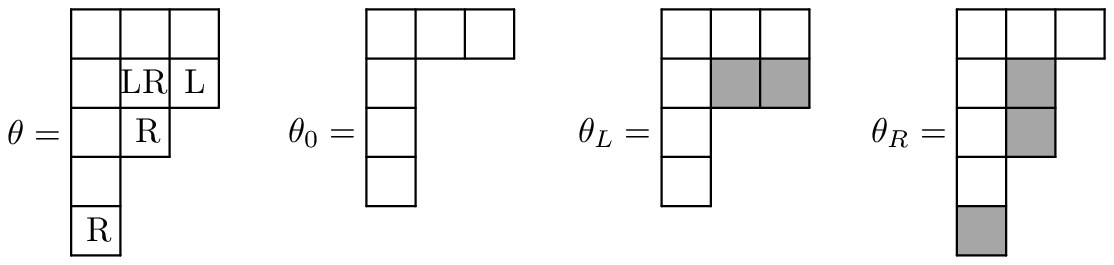}
\end{center}
Grey boxes correspond to the marked boxes in~$\theta$.

For~$\theta\in Y_{m,n}(w,a,b)$ let us denote
\begin{equation}
	\Sigma_\theta W^*=\Sigma_{\lambda(\theta)}U^*\otimes\Sigma_{\mu(\theta)}V^*
\end{equation}
for brevity.

\begin{prop}
	\label{Y_and_modules}
	There is an isomorphism~
	\[
		\Lambda^w(U^*\otimes V^*)\otimes\Sym^l(U^*)\otimes\Sym^r(V^*) =\bigoplus_{\theta\in Y_{m,n}(w,l,r)}\Sigma_{\omega(\theta)}W^*,
	\]
	where~$m=\dim U$ and~$n=\dim V$.
\end{prop}

\begin{remark}
	It is important that in the decomposition all summands have multiplicity~$1$.
\end{remark}

\begin{proof}
	By the well-known formula one has~$\Lambda^w(U^*\otimes V^*)=\bigoplus_{\mathrm{wt}(\nu)=w}\Sigma_\nu U^*\otimes\Sigma_{\nu'}V^*$~(see~\cite{Fulton}). 
	By the Pieri formula we get~$\Sigma_\nu U^*\otimes\Sigma_{\nu'}V^*\otimes\Sym^{l}U^*\otimes\Sym^{r}V^*=\bigoplus_{(\lambda,\mu)}\Sigma_\lambda U^*\otimes\Sigma_{\mu} V^*$, 
	where the sum is taken over all pairs of Young diagrams~$(\lambda,\mu)$ such that
	\begin{itemize}
		\item $\nu\subset\mu'$,~$\nu\subset\lambda$,~$|\lambda-\nu|=l$,~$|\mu'-\nu|=r$;
		\item the skew diagram~$\lambda-\nu$ has at most one box in each column the skew diagram~$\mu'-\nu$ has at most one box in each row.
	\end{itemize}

	Given such a pair~$(\lambda,\mu)$ we consider the Young diagram~$\theta=\lambda\cup\mu'$ and mark each box of~$\lambda-\nu$ with~`L' and each box of~$\mu'-\nu$ with~`R'
	(note that some boxes will be marked with both letters).
	Then~$\theta\in Y_{m,n}(w,l,r)$.
	Vice versa, if~$\theta\in Y_{m,n}(w,l,r)$, then~$\nu=\theta_0$,~$\lambda=\lambda(\theta)$ and~$\mu=\mu(\theta)$.
	It is clear that the properties above hold.
\end{proof}

\section{Isotypic components of the Koszul complex}

For any representation~$E$ of~$G=\GL(U)\times\GL(V)$ and any pair~$\omega=(\lambda,\mu)$ of Young diagrams denote by~$\mathrm{pr}_\omega(E)$ the projection to the isotypic component corresponding to~$\omega$.
Let~$K^\bullet$ be the Koszul complex
\begin{multline}
	\label{main_complex}
	\ldots \to
		\Lambda^{p+1}(U^*\otimes V^*)\otimes \Sym^{a+q-1}(U^*)\otimes \Sym^{b+q-1}(V^*) \to \\
	\to
		\Lambda^p(U^*\otimes V^*)\otimes \Sym^{a+q}(U^*)\otimes \Sym^{b+q}(V^*) \to \\
	\to
		\Lambda^{p-1}(U^*\otimes V^*)\otimes \Sym^{a+q+1}(U^*)\otimes \Sym^{b+q+1}(V^*) \to \ldots
\end{multline}
defined in~\eqref{K definition}. The set~$B$ is defined in~\ref{B definition}

\begin{prop}
	\label{cube_lemma}
	Let~$\omega=(\lambda,\mu)$ be a weight of~$\GL(U)\times\GL(V)$. 
	Then~$\mathrm{pr}_\omega(K^\bullet)=0$ unless~$\omega=\omega(\theta)$ for some~$\theta\in Y_{m,n}(w,l,r)$ such that~$a-b=l-r$, where~$m=\dim U$ and~$n=\dim V$.
	Moreover, in this case if~$\theta_L\cap\theta_R=\varnothing$, then we have 
	\[
		\mathrm{pr}_\omega(K^\bullet)\simeq\Sigma_\lambda U^*\otimes\Sigma_\mu V^*\otimes\Lambda^{\geqslant k}\Bbbk^n(-k),
	\]
	is a~$k$-th clipped combinatorial~$n$-cube, where~$n=|B(\lambda,\mu)|$ and~\[k=a-l=b-r.\]
\end{prop}

\begin{proof}
	Let us find conditions that are necessary for~$\mathrm{pr}_\omega(K^\bullet)\ne0$.
	By Proposition~\ref{Y_and_modules} there is a bijection between irreducible~$G$-modules in~$\Lambda^w W^*\otimes\Sym^{l}U^*\otimes\Sym^{r}V^*$ and~$Y_{m,n}(w,l,r)$.
	Hence, the complex~$K^\bullet$ is isomorphic to
	\begin{equation}
		\ldots\to
		\bigoplus_{\theta\in Y_{m,n}(p+1,a+q-1,b+q-1)}\Sigma_\theta W^* \to
		\bigoplus_{\theta\in Y_{m,n}(p,a+q,b+q)}\Sigma_\theta W^* \to
		\bigoplus_{\theta\in Y_{m,n}(p-1,a+q+1,b+q+1)}\Sigma_\theta W^* \to
		\ldots
	\end{equation}

	Note that~$|\theta_L|-|\theta_R|=(a+q+k)-(b+q+k)=a-b$ which gives the condition~$a-b=l-r$.

	Clearly,~$\mathrm{pr}_\omega(K^\bullet)$ is the subcomplex of~\eqref{main_complex} formed by summands~$\Sigma_\theta W^*$ with~$\omega(\theta)=\omega$.
	It remains to check that it is a clipped combinatorial cube.
	There exists a graded  vector space~$\mathcal{V}^\bullet$ such that~$\mathrm{pr}_\omega(K^\bullet)\simeq\Sigma_{\omega}W^*\otimes\mathcal{V}^\bullet$.
	Since the representation~$\Sigma_\omega W^*$ of the group~$G$ is irreducible, by the Schur Lemma~(see~\cite{FH})~$\End(\Sigma_\omega W^*)=\Bbbk$.
	Thus,
	\[
		d\in\End_G(\Sigma_\omega W^*\otimes\mathcal{V}^\bullet)=\End(\mathcal{V}^\bullet).
	\]
	
	Let us prove that the complex~$\mathrm{pr}_\omega(K^\bullet)$ is a clipped combinatorial cube.
	We will prove this in two steps.
	First, we will choose a basis in~$\mathcal{V}^\bullet$.
	Second, we will find all non-zero matrix elements of~$d$ in this basis.
	Then it will remain to use Lemma~\ref{clipped_cube_lemma}.
	\bigskip
	
	Take arbitrary~$\omega=(\lambda,\mu)$.
	Assume that~$\omega=\omega(\theta)$.
	Then~$\theta=\lambda\cup\mu'$.
	Moreover, all boxes of~$\lambda-\mu'$ should be marked with~`L',
	all boxes of~$\mu'-\lambda$ should be marked with~`R'
	and some of boxes of~$\lambda\cap\mu'$ may be marked with~`LR'.
	It is clear that one can mark with~`LR' only the boxes of the set~$B(\lambda,\mu)$ defined in~\ref{B definition}.
	Thus, the marked diagrams~$\theta$ such that~$\omega(\theta)=\omega$ are in bijection with the subsets of~$B(\lambda,\mu)$,
	i.\,e.~with the standard basis of the exterior algebra~$\Lambda^\bullet E$ of a vector space~$E$ spanned by vectors~$\xi_1,\ldots,\xi_n$
	numbered by elements of~$B(\lambda,\mu)$~--- the vector~$\xi_{i_1}\wedge\ldots\wedge\xi_{i_k}$ with~$1\leqslant i_1<\ldots<i_k\leqslant k$
	corresponds to~$\theta$ with boxes~$i_1,\ldots,i_k\in B(\lambda,\mu)$ marked with~`LR'.
	
	Take~$\theta\in Y_{m,n}(p,a+q,b+q)$ such that~$\omega(\theta)=\omega$.
	By Proposition~\ref{Y_and_modules} this diagram corresponds to a $G$-module in~$\Lambda^pW^*\otimes\Sym^{a+q}U^*\otimes\Sym^{b+q}V^*$, where~$\mathrm{wt}(\theta_0)=p$.
	The map~$\Lambda^p W^*\to\Lambda^{p-1}W^*\otimes W^*$ restricted to the summand~$\Sigma_{\theta_0}U^*\otimes\Sigma_{\theta_0'}V^*$ factors 
	as the sum of~$\Sigma_{\theta_0}U^*\to\Sigma_{\nu}U^* \otimes U^*$ and~$\Sigma_{\theta_0'}V^* \to \Sigma_{\nu'}V^* \otimes V^*$ for
	all subdiagrams~$\nu\subset\theta_0$ obtained by removing one box of~$\theta_0$.
	
	The map~$W^*\otimes\Sym^{a+q}U^*\otimes\Sym^{b+q}V^* \to \Sym^{a+q+1}U^*\otimes\Sym^{b+q+1}V^*$ on the summand corresponding to~$\nu$
	is the tensor product of the canonical maps~$U^*\otimes\Sym^{a+q}U^*\to\Sym^{a+q+1}U^*$ and~$V^*\otimes\Sym^{b+q}V^*\to\Sym^{b+q+1}V^*$.
	Finally, we see that the composition of the differential restricted to~$\Sigma_\theta W^*$ with the projection onto~$\Sigma_\vartheta W^*$ is non-zero
	if and only if~$\vartheta$ is obtained from~$\theta$ by marking one more box with~`LR'.
	
	By the Lemma~\ref{clipped_cube_lemma} the complex~$\mathrm{pr}_\omega(K^\bullet)$ is a $k$-th clipped combinatorial~$n$-cube.
	It remains to note that the component of maximal degree in this complex has~$l+|B(\lambda,\mu)|$ boxes containing~`L' and~$r+|B(\lambda,\mu)|$ boxes containing~`R'.
	Thus, this component has grading~$l+|B(\lambda,\mu)|-a=r+|B(\lambda,\mu)|-b=n-k$.
	Therefore we obtain~$\mathrm{pr}_\omega(K^\bullet)\simeq\Sigma_\omega W^*\otimes\Lambda^{\geqslant k}E(-k)$.
\end{proof}

\begin{example}
	The picture below gives an example of an isotypic component of the Koszul complex~\eqref{main_complex}.
	\begin{center}
		\includegraphics[scale=1]{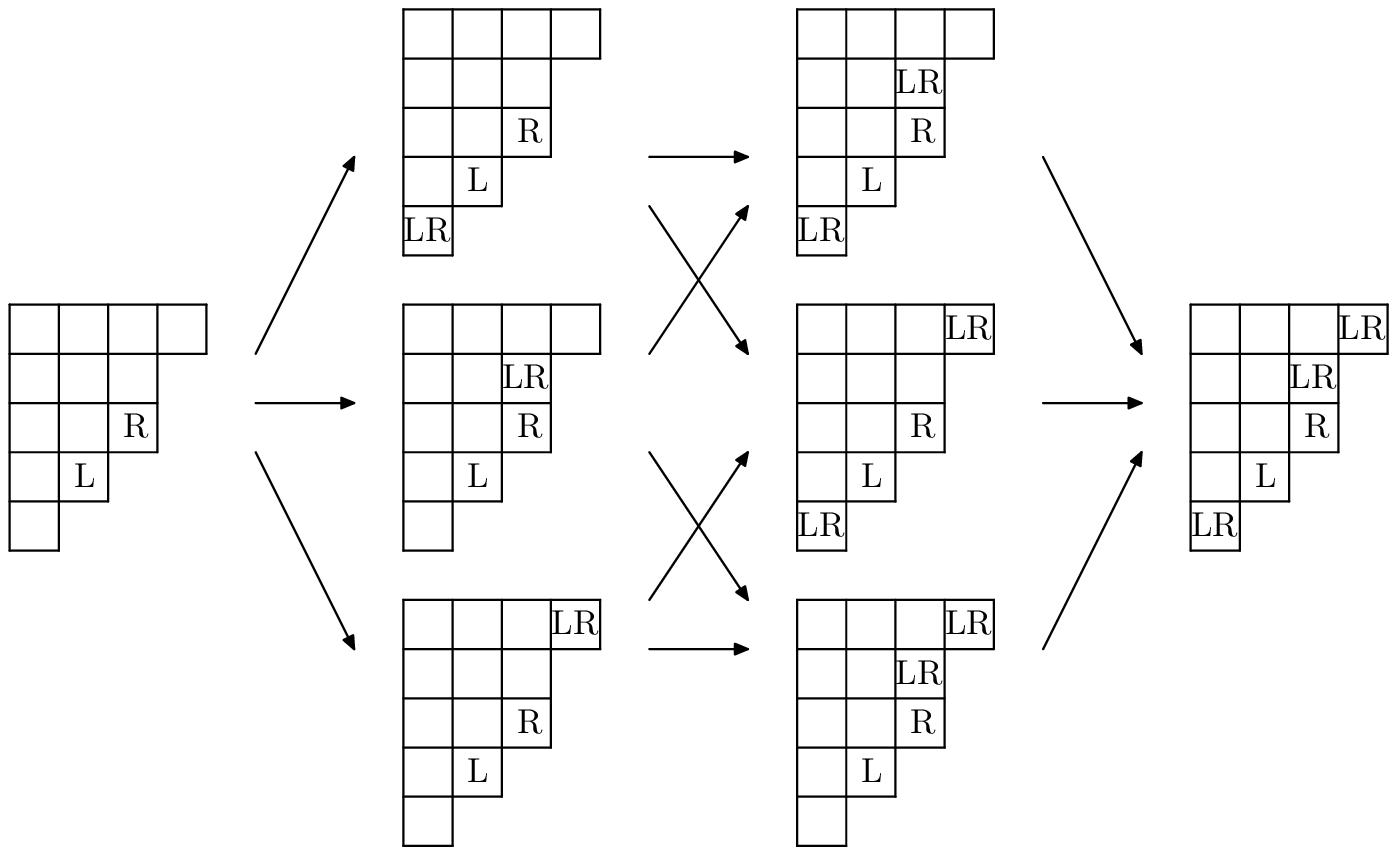}
	\end{center}
	Here
	\begin{center}
		\includegraphics[scale=1]{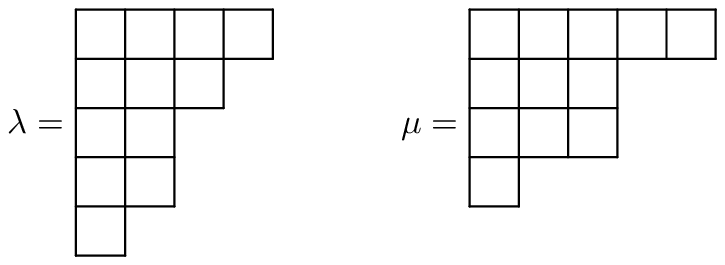}
	\end{center}

	Here~$k=-|\lambda-\mu'|=-|\mu'-\lambda|=-1$, and~$|B(\lambda,\mu)|=3$. On this picture all arrows are non-zero. Components of the weight~$\Sigma_{4,3,2,2,1}\otimes\Sigma_{5,3,3,1}$ appear in~$4$ degrees in the following $G$-modules:
	\[
	\begin{matrix}
		\Sigma_{4,3,2,1,1}U^*\otimes\Sym^1U^*\otimes\Sigma_{5,3,2,1}V^*\otimes\Sym^1V^*, & 
		\Sigma_{4,3,2,1}U^*\otimes\Sym^2U^*\otimes\Sigma_{4,3,2,1}V^*\otimes\Sym^2V^*, \\
		\Sigma_{4,2,2,1,1}U^*\otimes\Sym^2U^*\otimes\Sigma_{5,3,1,1}V^*\otimes\Sym^2V^*, &
		\Sigma_{3,3,2,1,1}U^*\otimes\Sym^2U^*\otimes\Sigma_{5,3,2}V^*\otimes\Sym^2V^*, \\
		\Sigma_{4,2,2,1}U^*\otimes\Sym^3U^*\otimes\Sigma_{4,3,1,1}V^*\otimes\Sym^3V^*, &
		\Sigma_{3,3,2,1}U^*\otimes\Sym^3U^*\otimes\Sigma_{4,3,2}V^*\otimes\Sym^3V^*, \\
		\Sigma_{3,2,2,1,1}U^*\otimes\Sym^3U^*\otimes\Sigma_{5,3,1}V^*\otimes\Sym^3V^*, &
		\Sigma_{3,2,2,1}U^*\otimes\Sym^4U^*\otimes\Sigma_{4,3,1}V^*\otimes\Sym^4V^*. \\
	\end{matrix}
	\]
	The complex is isomorphic to~\[\Sigma_{4,3,2,2,1}U^*\otimes\Sigma_{5,3,3,1}V^*\otimes\Lambda^\bullet\Bbbk^3(1).\]
\end{example}

We have described isotypic components of the Koszul complex.
It turns out that they are clipped combinatorial cubes.
These subcomplexes are acyclic iff the dimension of the cube is positive and the cube is not clipped.
Let us describe isotypic components which are $0$-cubes or $k$-th clipped cubes for~$k>0$.

\begin{defi}
	Denote~\[l^{a,b}(\lambda)=\max\{k\,|\,(a+k,b+k)\in\lambda\}+1.\]
\end{defi}

For example,~$l^{0,0}(\lambda)=l(\lambda)$ and~$l^{k,k}(\lambda)=l(\lambda)-k$, where~$l(\lambda)$ is the length of the diagonal~(see~\ref{wt and l definition}).

Two following Lemmas describe completely all~$0$-cubes in complex~\eqref{main_complex}.
In Lemma~\ref{0cubes_1} we construct combinatorial~$0$-cubes in the complex.
In Lemma~\ref{0cubes_2} we will show that the complexes described in the previous one are all~$0$-cubes in the complex.

\begin{lemma}
	\label{0cubes_1}
	Fix a Young diagram~$\nu$ and~$a,b\in\bbZ$ such that~$k=l^{a,b}(\nu)\geqslant0$.
	Then the isotypic component of the weight~$(\lambda,\mu)=(\extenddiag{\nu}{a+k},\extenddiag{\nu'}{b+k})$ in the Koszul complex~\eqref{main_complex} 
	is isomorphic to the combinatorial~$0$-cube~$\left(\Sigma_{\lambda}U^*\otimes\Sigma_{\mu}V^*\right)(k)$.
\end{lemma}

\begin{proof}
	Consider the last box~$c=(a+k-1,b+k-1)$ of diagonal of~$\nu$ starting at~$(a,b)$.
	Then there are~$a+k$ columns either containing~$c$ or to the left, and there are~$b+k$ rows either containing~$c$ or above.
	Let us add a box marked with~`L' to the end of each of first~$a+k$ columns and a box marked with~`R' to the end of each of first~$b+k$ rows.
	If~$c'=(x,y)\in\nu$ is on the diagonal or below, then there is a box containing~`L' below~$c'$.
	If~$c'=(x,y)\in\nu$ is on the diagonal or above, then there is a box containing~`R' to the right of~$c'$.
	Thus,~$B(\extenddiag{\nu}{a+k},\extenddiag{\nu'}{b+k})=\varnothing$, and this Lemma obviously follows from the Proposition~\ref{cube_lemma}.
\end{proof}

\begin{example}
	The picture below illustrates the construction above. Let us take a diagram~$\nu=(5,4,4,2)$ that is inside the bold line and~$(a,b)=(-2,-1)$. 
	Then the diagonal beginning at~$(a,b)$ is the set of boxes marked by dots and~$l=4$.
	We see that~$\lambda=(5,4,4,2,2)$ and~$\mu=(6,5,5,2)'=(4,4,3,3,3,1)$ and the isotypic component of the weight~$(\lambda,\mu)$ in~$K^\bullet$ equals~$(\Sigma_{5,4,4,2,2}U^*\otimes\Sigma_{4,4,3,3,3,1}V^*)(4)$.
	\begin{center}
		\includegraphics[scale=1]{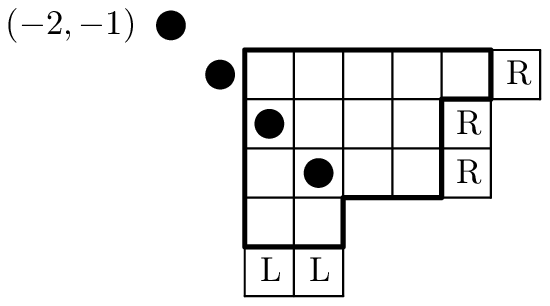}
	\end{center}
\end{example}

\begin{lemma}
	\label{0cubes_2}
	Let~$(\lambda,\mu)$ be a weight of~$G$ such that the isotypic component of the weight~$(\lambda,\mu)$ is a combinatorial~$0$-cube in the Koszul complex~\eqref{main_complex}.
	Then there is~$\nu$ satisfying assumptions of the Lemma~\ref{0cubes_1} such that~$\lambda=\extenddiag{\nu}{a+k}$ and~$\mu'=\extenddiag{\nu'}{b+k}$.
\end{lemma}

\begin{proof}
	Let~$\theta$ be the diagram~$\lambda\cup\mu'$, where the boxes~$\lambda-\mu'$ are marked with~`L' and the boxes~$\mu'-\lambda$ are marked with~`R'.
	By the construction we have~$\theta_L\cap\theta_R=\varnothing$.
	Since~$\mathrm{pr}_{(\lambda,\mu)}(K^\bullet)=\Sigma_\lambda U^*\otimes\Sigma_\mu V^*\otimes\Lambda^{\geqslant k}\Bbbk^n(-k)$ is a combinatorial~$0$-cube, from Proposition~\ref{cube_lemma} we get
	\begin{itemize}
		\item $n=|B(\lambda,\mu)|=0$;
		\item $k=a-|\lambda-\mu'|=b-|\mu'-\lambda|\leqslant 0$;
		\item there is at most one box of~$\lambda-\mu'$ in one column;
		\item there is at most box of~$\mu'-\lambda$ in one row.
	\end{itemize}

	By the last two conditions there are~$l=|\lambda-\mu'|$ columns with the letter~`L' and~$r=|\mu'-\lambda|$ rows with the letter~`R'.

	Denote by~$\xi_i=(x_i,y_i)$ for~$i=1,\ldots,s$ the boxes in~$\nu=\lambda\cap\mu'$ such that~$(x_i+1,y_i)\notin\nu$ and~$(x_i,y_i+1)\notin\nu$ (i.\,e.~the corner boxes of~$\nu$).
	Since~$x_i\ne x_j$ for~$i\ne j$, we can assume that~$x_1<x_2<\ldots<x_s$.
	Therefore~$y_1>y_2>\ldots>y_n$.
	The set~$B(\lambda,\mu)$ consists of boxes~$\xi_i=(x_i,y_i)$ such that~$(x_i,y_i+1)\notin\lambda-\mu'$ and~$(x_i+1,y_i)\notin\mu'-\lambda$.
	Since~$|B(\lambda,\mu)|=0$, for each~$\xi_i$ there is the letter~`L' in~$(x_i,y_i+1)$ or there is the letter~`R' in~$(x_i+1,y_i)$.

	Suppose that for some~$t$ we have~$(x_t+1,y_t)\in\theta_R$ and~$(x_{t+1},y_{t+1}+1)\in\theta_L$.
	Since~$\theta_0\cup\theta_L=\lambda$ and~$\theta_0\cup\theta_R=\mu'$ are Young diagrams, 
	the boxes above~$(x_t+1,y_t)$ and which are not in~$\theta_0$ are in~$\theta_R$ and the boxes to the left from~$(x_{t+1},y_{t+1}+1)$ and not in~$\theta_0$ are in~$\theta_L$.
	So we obtain that~$(x_t+1,y_{t+1}+1)\in\theta_L\cap\theta_R=\varnothing$.
	\begin{center}
		\includegraphics[scale=1]{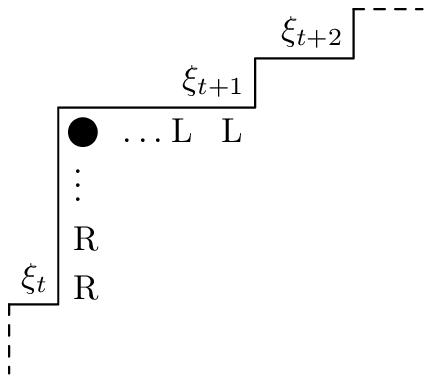}
	\end{center}
	On the picture above the box~$(x_t+1,y_{t+1}+1)$ is marked by a dot.
	If there is~`L' at~$(x_t,y_t+1)$, then there is~`L' at each column to the left of~$\xi_t$.
	If there is~`R' at~$(x_t+1,y_t)$, then there is~`R' at each row above~$(\xi_t)$.

	Hence, if~$(x_t+1,y_t)\in\theta_R$, then~$(x_{t'},y_{t'}+1)\notin\theta_L$ and~$(x_{t'}+1,y_{t'})\in\theta_R$ for~$t'>t$.
	In the same way if~$(x_t,y_t+1)\in\theta_L$, then~$(x_{t'}+1,y_{t'})\notin\theta_R$ and~$(x_{t'},y_{t'}+1)\in\theta_L$ for~$t'<t$.
	One of two following cases holds.
	\begin{itemize}
		\item There is~$t$ such that~$(x_t+1,y_t)\in\theta_R$ and~$(x_t,y_t+1)\in\theta_L$.
		\item There is~$t$ such that~$(x_t,y_t+1)\in\theta_L$,~$(x_{t+1},y_{t+1}+1)\notin\theta_L$,~$(x_t+1,y_t)\notin\theta_R$ and~$(x_{t+1}+1,y_{t+1})\in\theta_R$.
	\end{itemize}
	These cases are illustrated on the picture below.
	\begin{center}
		\includegraphics[scale=1]{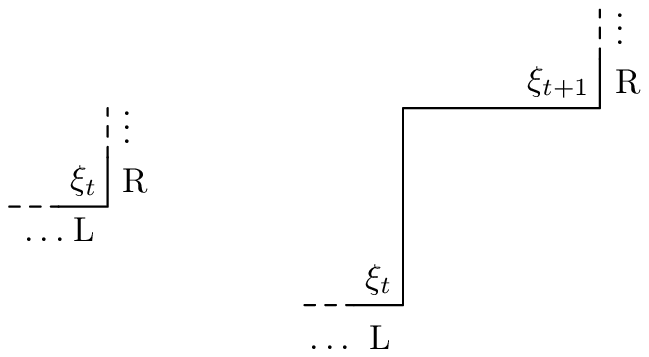}
	\end{center}
	
	Suppose the first case holds.
	Then for any~$t'\geqslant t$ one has~$(x_{t'}+1,y)\in\theta_R$ for~$y_{t'+1}<y\leqslant y_{t'}$ and~$y\geqslant0$ 
	(i.\,e.~there is a box containing~`R' at the end of each row above~$y_t$).
	For any~$t'\leqslant t$ one has~$(x,y_{t'}+1)\in\theta_L$ for~$x_{t'-1}<x\leqslant x_{t'}$ and~$x\geqslant0$
	(i.\,e.~there is a box containing~`L' at the end of each column to the left from~$x_t$).
	On the other hand,~$\theta_L$ and~$\theta_R$ can not have any other boxes.
	Otherwise we get a box in~$\theta_L\cap\theta_R$.
	Then~$(x_t,y_t)=(|\lambda-\mu'|,|\mu'-\lambda|)=(l,r)$ and for~$k=a-l=b-r$ we get~$l^{a,b}(\nu)=-k$.
	
	Suppose the second case holds.
	We see at least~$x_t$ boxes containing~`L' and at least~$y_{t+1}$ boxes containing~`R'.
	This means that~$l\geqslant x_t$ and~$r\geqslant y_{t+1}$.
	In one of these conditions equality holds.
	Otherwise the box~$(x_t+1,y_{t+1}+1)$ contains~`LR'.
	Suppose~$r=y_{t+1}$ (the second case is similar).
	If~$l>x_t$, then there are also~$l-x_t$ boxes with~`L' which can be only in the row~$y_{t+1}+1$.

	Obviously, in both cases we get~$\lambda=\extenddiag{\nu}{a+k}$ and~$\mu=\extenddiag{\nu'}{b+k}$.
\end{proof}

\begin{cor}
	\label{no_clipped}
	If~$a\leqslant 0$ and~$b\leqslant0$, then there are no isotypic components in the Koszul complex~\eqref{main_complex} that are $k$-th clipped combinatorial cubes for~$k>0$.
\end{cor}

\begin{proof}
	Assume that for some weight~$\omega=(\lambda,\mu)$ of~$G$ the isotypic component of this weight is a~$k$-th clipped combinatorial cube.
	Then it is~$\mathrm{pr}_\omega(K^\bullet)=\Sigma_\omega W^*\otimes\Lambda^{\geqslant k}\Bbbk^n(-k)\ne0$ for some~$n$.
	Then it has a non-zero component in degree~$0$.
	But this component lies in~$\Lambda^\bullet W^*\otimes\Sym^aU^*\otimes\Sym^bV^*$.
	Therefore this component has~$a$ letters~`L' and~$b$ letters~'R' in any~$\theta\in Y_{m,n}^{\omega}(w, a, b)$, where~$w=\mathrm{wt}(\lambda\cap\mu')$.
	Then~$a=b=0$ and~$\theta_L=\theta_R=\varnothing$.
	This means that~$k=0$.
\end{proof}

%\begin{theo}
%	\label{main_add}
%	Let the spaces~$R_{p,p+q}$ be the syzygy spaces of the Segre embedding~$\bbP(U)\times\bbP(V)\to\bbP(U\otimes V)$. Then there is an isomorphism of representations of $G$: 
%	\begin{equation}
%		\label{Rpq description}
%		R_{p,p+q} \cong \bigoplus_{\mathrm{wt}(\nu)=p,\atop l(\nu)=q}\left(\Sigma_{\extenddiag{\nu}{q}}U^*\otimes\Sigma_{\extenddiag{\nu'}{q}}V^*\right).
%	\end{equation}
%\end{theo}
Now let us prove Theorem~\ref{main_add_intr}.
\begin{proof}%[\ref{main_add_intr}]
	According to Corollary~\ref{last_complex} we need to calculate the cohomology groups of the complex~\eqref{main_complex} of representation of~$G$ in the case~$a=b=0$. 
	By Corollary~\ref{no_clipped} there are no clipped combinatorial cubes among isotypic components in the complex~\eqref{main_complex}.
	Then only combinatorial~$0$-cubes give a contribution to cohomology groups.
	By Lemmas~\ref{0cubes_1} and~\ref{0cubes_2} the isotypic component of a weight~$(\lambda,\mu)$ of~$G$ is a~$0$-cube in degree~$k$ iff there exists a diagram~$\nu$ such that~$\lambda=\extenddiag{\nu}{a+k}$, and~$\mu=\extenddiag{\nu}{b+k}$, and~$l^{a,b}(\nu)=k$.
	But~$a=b=0$, therefore~$l^{a,b}(\nu)=l(\nu)$.
\end{proof}

\begin{remark}
	The right-hand side of the equation of Theorem~\ref{main_add_intr} looks like being independent of~$m$ and~$n$. But~$\dim U^*=m$ and~$\dim V^*=n$. Therefore if~$\lambda$ has more than~$m$ rows or~$\mu$ has more than~$n$ rows, then~$U_\lambda^*\otimes V_\mu^*=0$.
\end{remark}

To write down a resolution for~$\mathscr{O}(a,b)$ we need another notation.

Denote by~$\widetilde{S}(w,a,b,k)$ the set of marked diagrams~$\theta$ such that~$\mathrm{wt}(\theta_0)+k=w$, ~$|\theta_L|=a$,~$|\theta_R|=b$ and~$|\theta_L\cap\theta_R|=k$.
If the isotypic component of a weight~$\omega$ in~\eqref{main_complex} is isomorphic to~$\Sigma_\theta W^*\otimes\Lambda^{\geqslant k}E(-k)$ for~$k>0$, a vector space~$E$ of dimension~$N$ and~$\omega=\omega(\theta)$, then~$\theta\in\widetilde{S}(w,a,b,k)$.
For any such a clipped cube there is a canonical element~$\theta$  such that each box lying in~$B(\theta)$ is marked by~`LR'.
Denote the set of such marked diagrams~$\theta$ by~$S(w,a,b,k)$.
Then a~$k$-th clipped combinatorial~$N$-cube~$\Sigma_\theta W^*\otimes\Lambda^{\geqslant k}E(-k)$ corresponds to the unique element~$\theta\in S(w,a,b,k)$, where~$N=B(\theta)$.

\begin{theo}
	\label{main_geom}
	Consider the Segre embedding~$X=\bbP(U)\times\bbP(V)\subset\bbP(U\otimes V)$.
	Let~$a\geqslant-m$ and~$b\geqslant-n$.
	Then there is a resolution
	\[
		\xymatrix{
			\ldots \ar[r] &
			\mathop{\bigoplus}\limits_{k\geqslant0}R_{1,k+1}^{a,b}\otimes\mathscr{O}(-1-k) \ar[r] &
			\mathop{\bigoplus}\limits_{k\geqslant0}R_{0,k}^{a,b}\otimes\mathscr{O}(-k) \ar[r] &
			\mathscr{O}_X(a,b) \ar[r] &
			0,
		}
	\]
	where~$m=\dim U$,~$n=\dim V$ and
	\begin{equation}
		\label{Oab syzygies}
		R_{p,p+q}^{a,b}=
		\begin{cases}
		\bigoplus_{l^{a,b}(\nu)=q \atop \mathrm{wt}(\nu)=p}U_{\extenddiag{\nu}{a+q}}^*\otimes V_{\extenddiag{\nu'}{b+q}}^*, & q > 0\text{ or } p=q=0; \\
		\bigoplus_{\theta\in S(a,b,p,k)}\left(U_{\lambda(\theta)}^*\otimes V_{\mu(\theta)}^*\right)^{\oplus\binom{|B(\theta)|-1}{k}}, & q = 0,\, p > 0. \\
		\end{cases}
	\end{equation}
\end{theo}

\begin{proof}
	Consider the standard diagonal resolution~\eqref{diag_resolution}.
	As in proof of Lemma~\ref{Koszul_lemma}, we construct a resolution for~$\mathscr{F}=\mathscr{O}_X(a,b)$.
	According to the Lemma~\ref{Koszul_lemma} for~$\mathscr{F}=\mathscr{O}_X(a,b)$ we get
	\[
		R_{p,p+q}^{a,b}=
		\mathrm{H}^q\left(\bbP^N,\Omega_{\bbP(W^*)}^p(p)\otimes i_*\mathscr{O}_X(a,b)\right)
		\cong
		\left(\Tor_p^S(\Bbbk,F(\mathscr{O}_X(a,b)))\right)_q,
	\]
	where we can calculate the right-hand side by the Koszul complex~\eqref{main_complex}.

	There are two types of isotypic components that contribute to cohomologies of the Koszul complex~$K^\bullet$~\eqref{main_complex}.
	They are combinatorial~$0$-cubes and clipped combinatorial cubes.
	By Lemmas~\ref{0cubes_1} and~\ref{0cubes_2} if the isotypic component of a weight~$(\lambda,\mu)$ is a combinatorial~$0$-cube and contributes to~$\mathrm{H}^q((K^\bullet)_p)$, then~$\nu=\lambda\cap\mu$,~$\mathrm{wt}(\nu)=p$,~$\lambda=\extenddiag{\nu}{q+l}$ and~$\mu'=\extenddiag{\nu'}{q+l}$, where~$l=l^{a,b}(\mu)$ and~$(K^\bullet)_p$ is~$p$-th graded component of~$K^\bullet$.
	This corresponds to the first line in~\eqref{Oab syzygies}.

	Now let us find clipped cubes.
	Each element~$\theta\in S(a,b,p,k)$ corresponds to the submodule~$\left(\Sigma_\theta W^*\right)^{\oplus\binom{|B(\theta)|-1}{k}}$ in the syzygy space~$R_{p,p}^{a,b}$.
\end{proof}

Since Theorem~\ref{main_geom} and all needed lemmas are functorial on vector spaces~$U$ and~$V$, we can apply the same reasoning to the relative situation.
This proves the following result generalizing~\ref{main_geom}.

\begin{theo}
	Let~$B$ be a smooth algebraic variety and~$\mathcal{U}$ and~$\mathcal{V}$ be vector bundles over~$B$.
	Consider the relative Segre embedding~$\mathcal{X}=\bbP_B(\mathcal{U})\times_B\bbP_B(\mathcal{V})\subset\bbP_B(\mathcal{U}\otimes\mathcal{V})$.
	Suppose~$a\geqslant -\mathrm{rk}(\mathcal{U})$ and~$b\geqslant -\mathrm{rk}(\mathcal{V})$.
	Then the sheaf~$\mathscr{O}_\mathcal{X}(a,b)$ has the following resolution:
	\[
		\xymatrix{
			\ldots \ar[r] &
			\mathop{\bigoplus}\limits_{k\geqslant0}\mathscr{R}_{1,k+1}^{a,b}\otimes\mathscr{O}(-1-k) \ar[r] &
			\mathop{\bigoplus}\limits_{k\geqslant0}\mathscr{R}_{0,k}^{a,b}\otimes\mathscr{O}(-k) \ar[r] &
			\mathscr{O}_\mathcal{X}(a,b) \ar[r] &
			0,
		}
	\]
	where
	\[
		\mathscr{R}_{p,p+q}^{a,b}=
		\begin{cases}
		\bigoplus_{l^{a,b}(\nu)=q \atop \mathrm{wt}(\nu)=p}\Sigma_{\extenddiag{\nu}{a+q}}\mathcal{U}^*\otimes\Sigma_{\extenddiag{\nu'}{b+q}}\mathcal{V}^*, & q > 0\text{ or } p=q=0; \\
		\bigoplus_{\theta\in S(a,b,p,k)}\left(\Sigma_{\lambda(\theta)}\mathcal{U}^*\otimes \Sigma_{\mu(\theta)}\mathcal{V}^*\right)^{\oplus\binom{|B(\theta)|-1}{k}}, & q = 0,\, p > 0. \\
		\end{cases}
	\]
\end{theo}

\begin{remark}
	\label{multiplication}
	There is a natural algebra structure on the direct sum of the syzygy spaces.
	Consider the Koszul complex~$K^\bullet$ of the Segre embedding (defined in~\ref{main_complex}) and choose equivariant projection~$p$ and embedding~$i$:
	$$
	\xymatrix{
		\mathcal{A}=K^\bullet \ar@<0.3em>[r]^p &
		\mathrm{H}^\bullet(K^\bullet) \ar@<0.3em>[l]^i,
	}
	$$
	such that~$p\circ i=\mathrm{Id}$.
	Denote by~$\pi$ the multiplication on the Koszul complex.
	Then the multiplication on the sum of the syzygy spaces is given by the formula
	\[
		\varpi(x,y) = p(\pi(i(x),i(y))).
	\]
	
	Denote~$Y=\bigcup_{p,q}Y_{m,n}(p,q,q)$.
	Then~$\mathcal{A}=\bigoplus_{\theta\in Y}\Sigma_\theta W^*$.
	Note that each $G$-submodule in the direct sum of the syzygy spaces has multiplicity~$1$.
	Also each submodule in the algebra~$\Lambda^\bullet W^*$ has multiplicity~$1$.
	Hence, the multiplication
	\[
		\varpi\colon\bigoplus_{\theta_1}\Sigma_{\theta_1}W^*\otimes\bigoplus_{\theta_2}\Sigma_{\theta_2}W^*\to\bigoplus_{\theta_3}\Sigma_{\theta_3}W^*
	\]
	is given by the structure constants~$\varpi_{\theta_1,\theta_2}^{\theta_3}$.
	
	Take three marked diagrams~$\theta_1,\theta_2,\theta_3\in Y$.
	From Theorem~\ref{main_add_intr} it implies that
	$
		\omega(\theta_i)=(\extenddiag{\nu_i}{s_i},\extenddiag{\nu_i'}{s_i}),
	$
	where~$s_i=l(\nu_i)$.
	Then~$\varpi_{\theta_1,\theta_2}^{\theta_3}\ne0$ if and only if~$s_1+s_2=s_3$ and~$\nu_3\subseteq\nu_1\otimes\nu_2$ in the exterior algebra~$\Lambda^\bullet W^*$.
	(Recall that~$\Lambda^\bullet W^*=\bigoplus_\nu U_\nu^*\otimes V_{\nu'}^*$.)
	
	Obviously if~$s_3\ne s_1+s_2$ or~$\nu_3\not\subseteq\nu_1\otimes\nu_2$, then we obtain~$\varpi_{\theta_1,\theta_2}^{\theta_3}=0$ because~$\varpi$ is homogeneous and equivariant.
	Suppose~$s_3= s_1+s_2$ and~$\nu_3\subseteq\nu_1\otimes\nu_2$.
	Since the $G$-submodule~$\Sigma_{\theta_3}W^*$ lies in~$\varpi(\Sigma_{\theta_1}W^*,\Sigma_{\theta_2}W^*)$, has multiplicity~$1$ in the Koszul complex~$K^\bullet$ and in the syzygies, the projection~$p$ does not annihilate it.
	Therefore~$\varpi_{\theta_1,\theta_2}^{\theta_3}\ne0$.
\end{remark}

\appendix
\section{Examples of syzygies}

%\subsection{Syzygies}

\begin{example}
Consider the Segre embedding~$X=\bbP^1\times\bbP^2\subset\bbP^5$.
There are three diagrams with width at most $2$ and height at most~$1$.
The syzygy algebra consists of three~$\GL(2)\times\GL(3)$-modules.
\begin{center}
\includegraphics[scale=1]{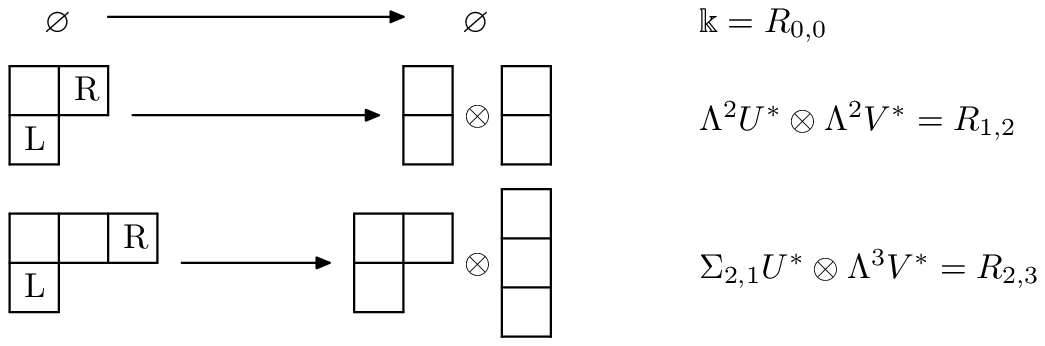}
\end{center}

We get~$R_{1,2}=\Lambda^2\Bbbk^2\otimes\Lambda^2\Bbbk^3$,~$R_{2,3}=\Sigma_{2,1}\Bbbk^2\otimes\Lambda^3\Bbbk^3$, the syzygy algebra has the zero multiplication.

In particular, we get the following resolution:
\[
	\xymatrix{
		0 \ar[r] &
		\mathscr{O}(-2)^2 \ar[r] &
		\mathscr{O}(-1)^3 \ar[r] &
		\mathscr{O} \ar[r] &
		\mathscr{O}_X \ar[r] &
		0.
	}
\]
\end{example}

\begin{example}
Consider the Segre embedding~$X=\bbP^2\times\bbP^3\subset\bbP^{11}$.
There are $10$ diagrams with width at most~$3$ and height at most~$2$. The syzygy algebra consists of ten~$\GL(3)\times\GL(4)$-modules.
\begin{center}
\includegraphics[scale=0.8]{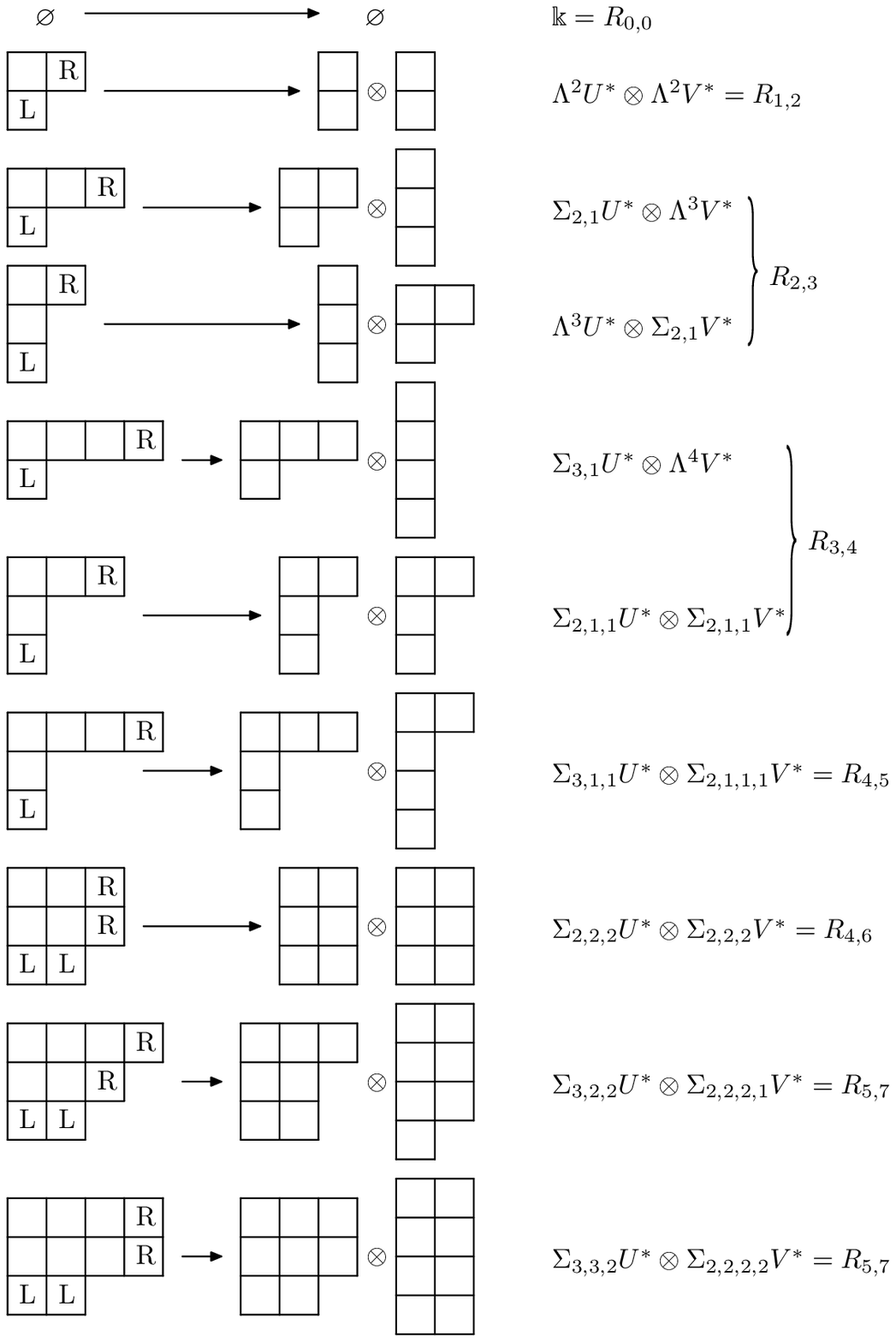}
\end{center}

We get 
$$R_{1,2}=\Lambda^2\Bbbk^3\otimes\Lambda^2\Bbbk^4,$$ 
$$R_{2,3}=\Sigma_{2,1}\Bbbk^3\otimes\Lambda^3\Bbbk^4\oplus\Lambda^3\Bbbk^3\otimes\Sigma_{1,2}\Bbbk^3,$$ 
$$R_{3,4}=\Sigma_{3,1}\Bbbk^3\otimes\Lambda^4\Bbbk^4\oplus\Sigma_{2,1,1}\Bbbk^3\otimes\Sigma_{2,1,1}\Bbbk^4,$$ 
$$R_{4,5}=\Sigma_{3,1,1}\Bbbk^3\otimes\Sigma_{2,1,1,1}\Bbbk^4$$
$$R_{4,6}=\Sigma_{2,2,2}\Bbbk^3\otimes\Sigma_{2,2,2}\Bbbk^4,$$ 
$$R_{5,7}=\Sigma_{3,2,2}\Bbbk^3\otimes\Sigma_{2,2,2,1}\Bbbk^4,$$ 
$$R_{6,8}=\Sigma_{3,3,2}\Bbbk^3\otimes\Sigma_{2,2,2,2}\Bbbk^4.$$

There are only the following non-zero multiplication maps:
$$R_{1,2}\times R_{3,4}\to R_{4,6},\quad
  R_{1,2}\times R_{4,5}\to R_{5,7},\quad
  R_{2,3}\times R_{2,4}\to R_{4,6},$$ $$
  R_{2,3}\times R_{3,5}\to R_{5,7},\quad
  R_{2,3}\times R_{4,6}\to R_{6,8},\quad
  R_{3,4}\times R_{3,4}\to R_{6,8}.$$
  
In particular, we get the following resolution:
\begin{equation*}
	0 \to
	\mathscr{O}(-7)^{3} \to
	\mathscr{O}(-6)^{12} \to %\\ \to
	\mathscr{O}(-5)^{12} \oplus
	\mathscr{O}(-4)^{24} \to
	\mathscr{O}(-3)^{60} \to %\\ \to
	\mathscr{O}(-2)^{52} \to
	\mathscr{O}(-1)^{18} \to
	\mathscr{O} \to
	\mathscr{O}_X \to 
	0.
\end{equation*}

\end{example}

\section{Examples of resolutions of sheaves}

\begin{example}
Consider the Segre embedding~$X=\bbP^1\times\bbP^2\subset\bbP^5$ and the sheaf~$\mathscr{O}(-1,1)$.
Then there are $6$ isotypic components that are~$0$-cubes and there are no clipped cubes.
\begin{center}
\includegraphics[scale=1]{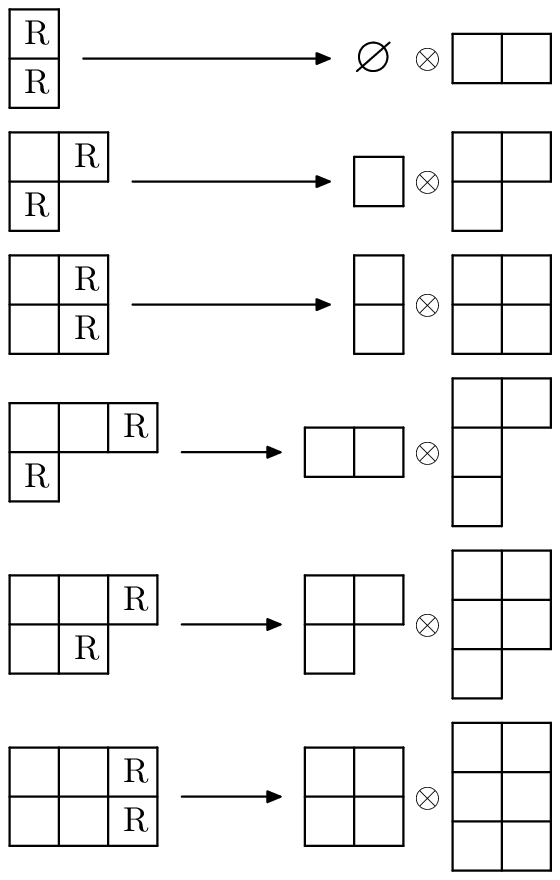}
\end{center}

We get the following resolution:
\begin{equation*}
	0 \longrightarrow
	\mathscr{O}(-5)\longrightarrow
	\mathscr{O}(-4)^{6}\longrightarrow %\\ \to
	\mathscr{O}(-3)^{15}\longrightarrow %\\ \to
	\mathscr{O}(-2)^{16} \longrightarrow
	\mathscr{O}(-1)^{6} \longrightarrow
	\mathscr{O}_X(-1,1) \longrightarrow
	0.
\end{equation*}
\end{example}

\bibliographystyle{amsplain}

\end{document}